\documentclass[11pt]{article}
\usepackage[english]{babel}
\usepackage[latin1]{inputenc}
\usepackage[T1]{fontenc}

\usepackage{graphics,graphicx} 

\usepackage{latexsym}
\usepackage{amsfonts}
\usepackage{amsthm}
\usepackage{amsmath}
\usepackage{eepic,epic}
\usepackage{graphicx}
\usepackage[hmargin=2.3cm,vmargin=2.5cm]{geometry}

\usepackage{numprint}

\newcommand\p{\circle*{0.2}}

\newtheorem{theorem}{Theorem}
\newtheorem{proposition}{Proposition}
\newtheorem{conjecture}{Conjecture}

\title{S-crucial and bicrucial permutations with respect to squares}
\author{Ian Gent\thanks{School of Computer Science, University of St Andrews, St Andrews, Fife KY16 9SX, UK. \newline
Emails: \{ian.gent, alexander.konovalov, sl4, pwn1\}@st-andrews.ac.uk
}, Sergey Kitaev\thanks{School of Computer and Information Sciences, University of Strathclyde, Glasgow, G1 1HX, UK. \newline Email: sergey.kitaev@cis.strath.ac.uk}, Alexander Konovalov\footnotemark[1], \\ Steve Linton\footnotemark[1], and Peter Nightingale\footnotemark[1]}

\begin{document}

\maketitle
\thispagestyle{empty}

\begin{abstract}  A permutation is square-free if it does not contain two consecutive factors of length two or more that are order-isomorphic. A permutation is bicrucial with respect to squares if it is square-free but any extension of it to the right or to the left by any element gives a permutation that is not square-free. 

Bicrucial permutations with respect to squares were studied by Avgustinovich et al. \cite{AKPV}, who proved that there exist bicrucial permutations of lengths $8k+1, 8k+5, 8k+7$ for $k\ge 1$. It was left as open questions whether bicrucial permutations of even length, or  such permutations of length $8k+3$ exist. In this paper, we provide an encoding of orderings which allows us, using the constraint solver Minion, to show that bicrucial permutations of even length exist, and the smallest such permutations are of length 32. To show that 32 is the minimum length in question, we establish a result on left-crucial (that is, not extendable to the left) square-free permutations which begin with three elements in monotone order. Also, we show that bicrucial permutations of length $8k+3$ exist for $k=2,3$ and they do not exist for $k=1$. 

Further, we generalise the notions of right-crucial, left-crucial, and bicrucial permutations studied in the literature in various contexts, by introducing the notion of $P$-crucial permutations that can be extended to the notion of $P$-crucial words. In S-crucial permutations, a particular case of $P$-crucial permutations, we deal with permutations that avoid prohibitions, but whose extensions in any position contain a prohibition. We show that S-crucial permutations exist with respect to squares, and minimal such permutations are of length 17.

Finally, using our software, we generate much of relevant data showing, for example, that there are 162,190,472 bicrucial square-free permutations of length 19. \\

\end{abstract}

\section{Introduction}\label{intro}

A {\em factor} of a word is a number of consecutive letters in the word. A word $w$ {\em avoids} a word $u$ if $w$ does not contain $u$ as a factor.  Let $S$ be a set of prohibited factors, that is, factors to be avoided. A word $w$ over $A$ is {\em right-crucial} (resp., {\em left-crucial}) with respect to $S$ if it avoids the prohibitions, but adjoining a new letter from $A$ to the right (resp., left) of $w$ gives a word that does not avoid the prohibitions. Clearly, studying right-crucial words can be turned to studying left-crucial words, and vice versa (through reversing all words in question). A word is {\em bicrucial} if it is both right- and left-crucial. 

We say that a word $w$ contains a {\em $k$-th
power}, if $w$ contains a factor $XX\cdots X$ with $k$ non-empty
words $X$. The case $k=2$ corresponds to {\em squares} in words, and their study was initiated by Alex Thue in \cite{thue} in 1906. A word $w$ contains an {\em abelian $k$-th power},
if $w$ contains a factor $X_1X_2\cdots X_k$ where $X_i$ is a
permutation of $X_1$ for $2\leq i\leq k$. Paul Erd\H{o}s \cite{pE61some} introduced the notion of abelian squares (the case of $k=2$) in 1961. Right-crucial and bicrucial permutations with respect to squares, abelian squares, and, more generally, $k$-th powers and abelian $k$-th powers were studied in \cite{AGHK,EK,B, C, GHK, K}. These studies belong to the area of combinatorics on words. 

Avgustinovich et al. \cite{AKPV} extended the notion of squares, as well as the notion of (right-,bi)crucial words, from words to permutations by merging the respective notions in combinatorics on words and the theory of permutation patterns (see \cite{Kit} for a comprehensive introduction to the theory); precise definitions will be given below. More studies in this direction were conducted in \cite{AKV}. 

It was shown in  \cite{AKPV} that right-crucial permutations exist of any length larger than $6$, while existence of bicrucial permutations was only shown for lengths $8k+1, 8k+5, 8k+7$ for $k\ge 1$ (the shortest such permutation is of length $9$). The question on whether there exists bicrucial permutations of even lengths and of length $8k+3$ for some $k\geq 1$ was open for about three years. In this paper, we will show that bicrucial permutations of even length exist (the smallest such permutation is of length 32), and that bicrucial permutations of length $8k+3$ exist for $k=2,3$ and they do not exist for $k=1$. Our main tool in obtaining the results is an {\em encoding of orderings} (see Subsection \ref{encoding-orderings}) and the constraint solver {\em Minion}.

Note that what we call ``right-crucial permutations'' are known as ``crucial permutations'' in the literature. In this paper we generalise the notion of  right-crucial and bicrucial permutations to {\em $P$-crucial permutations} with respect to a given set of prohibitions, where $P$ is a possibly infinite set of non-negative integers. A particular example of $P$-crucial permutations is {\em S-crucial} (standing for {\em Super-crucial}) {\em permutations} which avoid prohibitions but extending them in {\em any} position leads to a prohibition. This notion has an obvious counter-part in the case of words. We will show that the minimum S-crucial permutation with respect to squares is of length 15 (see Section~\ref{P-crucial-S-crucial}).

The paper is organised as follows. In Section~\ref{sec2}  we discuss all objects of interest stating known, and our results on them. In Section \ref{generating-perms} we discuss software used to obtain most of our results. Finally, in Section~\ref{open}  we discuss some of directions for further research.  

In this paper, we need the following notions. The {\em reverse} of a permutation $\pi=\pi_1\pi_2\cdots \pi_n$ is the permutation $r(\pi)=\pi_n\pi_{n-1}\cdots\pi_1$, while the {\em complement} of $\pi$ is the permutation $c(\pi)=(n+1-\pi_1)(n+1-\pi_2)\cdots (n+1-\pi_n)$. For example, if $\pi=2134$ then $r(\pi)=4312$ and $c(\pi)=3421$. When we say ``{\em up to symmetry}'' in counting problems here we mean that if a permutation $\pi$ has been already counted, then the permutation $c(\pi)$ is not considered and (for properties where $\pi$ having the property guarantees that $r(\pi)$ does), $r(\pi)$ and $c(r(\pi))=r(c(\pi))$ are not to be considered.  

\section{Objects of interest and our results}\label{sec2}

In what follows, for a permutation $\pi=\pi_1\pi_2\cdots\pi_n$ in question, we will write $x\ldots y=s\ldots t$ to indicate that the pattern formed by the elements of $\pi$ in positions $x,x+1,\ldots,y$  is equal to that formed by the elements of $\pi$ in position $s,s+1,\ldots,t$. In $x\pi'$, an extension of $\pi$ to the left (defined properly below), we assume $x$ to be in position 0, while in $\pi'y$, an extension of $\pi$ to the right (again, defined properly below), we assume $y$ to be in position $n+1$. For example, in the permutation 492517638 we have 12=34 (because 49 and 25 form the same pattern 12) and 456=789, which can be written $4\ldots 6=7\ldots 9$ (because 517 and 638 form the same pattern 213). 

We will be referring to tables showing the number of permutations satisfying various properties. 
These were calculated using constraint programming techniques and a constraint solver called Minion \cite{minion}.
We give more detail about our approach below, in Section \ref{sec:cp}. 
For now we just mention that
we report a metric of search cost called  nodes. 
Minion (along with many other constraint solvers) performs a depth-first exploration of a rooted binary tree. 
The root represents the initial state and left branches represent an (exploratory) assignment of a variable. Solutions are found at leaf nodes. Minion reports the number of left branches as its node count.

\subsection{Square-free permutations}\label{square-free-perms} 

We use the one-line notation to represent permutations. A {\em factor} of a permutation $\pi$ is one or more consecutive elements in $\pi$. For example, 4273, 36 and 1 are factors in the permutation 5142736.

A permutation is {\em square-free} if it does not contain two consecutive factors of length two or more that are in the same relative order, that is, that are order-isomorphic. For example, the permutation 243156 is square-free, while the permutation 631425 contains the square 3142 (indeed, 31 is order isomorphic to 42). Another example of a permutation that does not avoid squares is 742563891 as it contains the squares 425638, 5638 and 563891. The number of square-free permutations of length $n$, for $n=1,2,\ldots, 18$, is 
$$\begin{array}{c}1, 2, 6, 12, 34, 104, 406, 1112, 3980, 15216, 68034, 312048, 1625968, 8771376, \\
53270068, 319218912, 2135312542, 14420106264\end{array}$$
and this is the sequence  A221989 in the {\em On-Line Encyclopedia of Integer Sequences}, {\em OEIS} \cite{OEIS}. See also Table~\ref{tab:squarefree}. It is known \cite{AKPV} that  the number of square-free permutations
of length $n$ is $n^{n(1-\varepsilon_n)}$ where
$\varepsilon_n\rightarrow
0$ when $n\rightarrow\infty$. 

\def\tabheader{
\begin{table}
\begin{center}
\begin{tabular}{r|r|r|r|r}
$n$ & Num Solutions & Nodes & Search time (s) & Total time (s) \\ \hline
}
\def\tabheader{
\begin{table}
\begin{center}
\begin{tabular}{r|rr|rr|rr}
$n$ & Solutions & Nodes & Solutions  & Nodes & Solutions & Nodes \\ 
 &  & Searched &up to Symmetry & Searched & $\pi=r(c(\pi))$& Searched \\

 &  & & & & up to Symmetry& \\ \hline
 }
\def\tabheadershort{
\begin{table}
\begin{center}
\begin{tabular}{r|r|r|r}
$n$ & Num Solutions & Nodes & Total time (s) \\ \hline
}
\def\tabheadershort{
\begin{table}
\begin{center}
\begin{tabular}{r|rr|r|rr}
$n$ & Solutions & Nodes & Solutions  & Solutions & Nodes \\ 
  &  & Searched &up to Symmetry & $\pi=r(c(\pi))$& Searched \\
 &  & & & up to Symmetry& \\ \hline
}
\newcommand\tabfooter[2]{
\end{tabular}
\end{center}
\caption{#1}
\label{#2}
\end{table}
}
\newcommand\Comment[1]{}

\tabheader
 3 &            6 &           11 &            2 &            4 &            1 &            2 \\
 4 &           12 &           29 &            3 &           11 &            0 &            3 \\
 5 &           34 &           90 &           10 &           33 &            3 &            9 \\
 6 &          104 &          296 &           26 &           93 &            0 &           17 \\
 7 &          406 &         1147 &          105 &          276 &            7 &           23 \\
 8 &         1112 &         3845 &          278 &          932 &            0 &           10 \\
 9 &         3980 &        14505 &         1011 &         3936 &           32 &          117 \\
10 &        15216 &        57680 &         3804 &        14534 &            0 &          130 \\
11 &        68034 &       256501 &        17065 &        60237 &          113 &          402 \\
12 &       312048 &      1209914 &        78012 &       313175 &            0 &           21 \\
13 &      1625968 &      6244642 &       406795 &      1764062 &          606 &         2913 \\
14 &      8771376 &     34262121 &      2192844 &      8714313 &            0 &          434 \\
15 &     53270068 &    204080489 &     13318687 &     47248481 &         2340 &        11593 \\
16 &    319218912 &   1260657446 &     79804728 &    318506973 &            0 &           36 \\
17 &  {2135312542}& - &    533838106 &   2336505587 &        19941 &       107779 \\
\tabfooter{Enumerating square-free permutations of length $n$ using the constraint programming approach described in
Section~\ref{sec:cp}.   Every number is by  direct computer enumeration except for $n=17$ solutions, which is calculated as 4 times the number of solutions up to symmetry minus twice the number of reverse-complemented solutions.    The first column of solutions is A221989 in \cite{OEIS}.}{tab:squarefree}

It is easy to see \cite{AKPV}  that for a permutation $\pi=\pi_1\pi_2\cdots$ to avoid squares of length 2, there must exist $i\in \{0,1,2,3\}$ so that for
every non-negative integer $t$, the inequalities
$$\pi_{i+4t}< \pi_{i+4t\pm 1} \mbox{ and } \pi_{i+4t+2}> \pi_{i+4t\pm 3} \eqno(1)$$
hold.
Schematically, the four possible kinds of square-free permutations (according to the choice of $i$) are as follows:

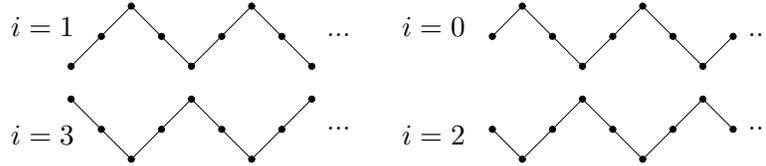
\begin{center}
\begin{picture}(24,3)
\put(2,0){\put(-2,1){$i=1$}\put(0,0){\p} \put(1,1){\p} \put(2,2){\p}
\put(3,1){\p}\put(4,0){\p}\put(5,1){\p}\put(6,2){\p}\put(7,1){\p}\put(8,0){\p}\put(8.5,1){...}
\path(0,0)(1,1)(2,2)(3,1)(4,0)(5,1)(6,2)(7,1)(8,0)

\put(11,1){$i=0$}\put(14,1){\p} \put(15,2){\p}
\put(16,1){\p}\put(17,0){\p}\put(18,1){\p}\put(19,2){\p}\put(20,1){\p}\put(21,0){\p}\put(22,1){\p}\put(22.5,1){...}
\path(14,1)(15,2)(16,1)(17,0)(18,1)(19,2)(20,1)(21,0)(22,1)

}\end{picture}

\begin{picture}(24,3)
\put(2,0){\put(-2,0.5){$i=3$}\put(0,2){\p}
\put(1,1){\p}\put(2,0){\p}\put(3,1){\p}\put(4,2){\p}\put(5,1){\p}\put(6,0){\p}\put(7,1){\p}\put(8,2){\p}\put(8.5,1){...}
\path(0,2)(1,1)(2,0)(3,1)(4,2)(5,1)(6,0)(7,1)(8,2)

\put(11,0.5){$i=2$}\put(14,1){\p}\put(15,0){\p}\put(16,1){\p}\put(17,2){\p}\put(18,1){\p}\put(19,0){\p}\put(20,1){\p}\put(21,2){\p}
\put(22,1){\p}\put(22.5,1){...}
\path(14,1)(15,0)(16,1)(17,2)(18,1)(19,0)(20,1)(21,2)(22,1)

}\end{picture}
\end{center}

\noindent
where a dot represents an element in $\pi$ and the order of elements
represented by two non-consecutive dots is irrelevant, whereas each
pair of consecutive dots is comparable (a lower dot represents a
smaller element).

For better understanding of the properties of square-free permutations, we enumerated such permutations up to symmetry. These values are presented in Table~\ref{tab:squarefree}. These sequences were not known to the OEIS \cite{OEIS}, and to the best of our knowledge, these results are new. We note that solutions up to symmetry were used in Table~\ref{tab:squarefree} to calculate the total number of square-free permutations in the case of $n=17$. The inclusion-exclusion principle is used here: each symmetrically distinct $\pi$ gives four permutations through combinations of $c$ and $r$, unless $\pi = c(r(\pi))$ (equivalently, $\pi=r(c(\pi))$), in which case it gives two.   Also, note that the 0s in the next to the rightmost column in Table~\ref{tab:squarefree} follow from the zigzag structure of square-free permutations.

\subsection{Right- and left-crucial permutations with respect to squares}\label{sub-right-crucial-perm}

In this subsection we define the notion of {\em right-crucial permutations with respect to squares} which was previously known \cite{AKPV}  as {\em crucial permutations with respect to squares}.  

Let $\pi=\pi_1\pi_2\cdots\pi_n$ be a permutation of length $n$. An {\em extension} of $\pi$ by an element $x\in\{1,2,\ldots,n+1\}$ in position $i\in\{0,1,\ldots,n\}$ is the permutation $\pi'_1\pi'_2\cdots\pi'_{i}x\pi'_{i+1}\pi'_{i+2}\cdots\pi'_n$ of length $n+1$, where $\pi'_i=\pi_i$ if $\pi_i<x$ and $\pi'_i=\pi_i+1$ otherwise. In particular, the case of $i=0$ is called an extension of $\pi$ by $x$ to the left and the case $i=n$ is called an extension of $\pi$ by $x$ to the right. 

A permutation is right-crucial with respect to squares if it is square-free but any extension of it to the right by any element gives a permutation that is not square-free. We also say that such a permutation is {\em right-crucial square-free permutation}. For example, the permutation 2136547 is right-crucial with respect to squares. It is shown in \cite{AKPV} that right-crucial permutations with respect to squares exist of any length larger than 6. All the right-crucial
permutations with respect to squares of length 7 are listed below: 
$$\begin{array}{l}2136547,\ 2137546,\ 2146537,\ 2147536,\ 2156437,\ 2157436,\
2167435,\ \\
 3146527,\ 3147526,\ 3156427,\ 3157426,\ 3167425,\ 3246517,\
 3247516,\ \\
 3256417,\ 3257416,\ 3267415,\ 3421675,\ 3521674,\ 3621574,\
 3721564,\ \\
 4156327,\ 4157326,\ 4167325,\ 4256317,\ 4257316,\ 4267315,\
 4356217,\ \\
 4357216,\ 4367215,\ 4521673,\ 4531672,\ 4532671,\ 4621573,\
 4631572,\ \\
 4632571,\ 4721563,\ 4731562,\ 4732561,\ 5167324,\ 5267314,\
 5367214,\ \\
 5467213,\ 5621473,\ 5631472,\ 5632471,\ 5641372,\ 5642371,\
 5721463,\ \\
 5731462,\ 5732461,\ 5741362,\ 5742361,\ 6721453,\ 6731452,\
 6732451,\ \\
 6741352,\ 6742351,\ 6751342,\ 6752341.\end{array}$$

The number of right-crucial permutations of length $n$ with respect to squares, for $n=7,8,\ldots, 17$, is 
$$60, 140, 518, 1444, 8556, 31992, 220456, 984208, 7453080, 39692800, 289981136$$
and this is the sequence  A221990 in the OEIS \cite{OEIS}. See also Table \ref{tab:right-crucial}.

\tabheadershort
 3 &            0 &            0 &            0 &            0 &            0 \\
 4 &            0 &            0 &            0 &            0 &            0 \\
 5 &            0 &            0 &            0 &            0 &            0 \\
 6 &            0 &            0 &            0 &            0 &            0 \\
 7 &           60 &          262 &           30 &            0 &           16 \\
 8 &          140 &          658 &           70 &            0 &           25 \\
 9 &          518 &         2978 &          259 &            5 &           60 \\
10 &         1444 &         9135 &          722 &            0 &          232 \\
11 &         8556 &        44110 &         4278 &            0 &           46 \\
12 &        31992 &       157334 &        15996 &            0 &           61 \\
13 &       220456 &      1109525 &       110228 &          168 &         1841 \\
14 &       984208 &      5008522 &       492104 &            0 &          845 \\
15 &      7453080 &     46370720 &      3726540 &            0 &         4113 \\
16 &     39692800 &    251929277 &     19846400 &            0 &          113 \\
17 &    289981136 &   1939299692 &    144990568 &         3522 &        81951 \\
\tabfooter{Enumerating right-crucial permutations of length $n$ with respect to squares.  In this case, $\pi$ being a solution only guarantees that $c(\pi)$ is, but not $r(\pi)$.  The solutions up to symmetry is always exactly half the total number of solutions, and that number is used in the centre column without an additional search. }{tab:right-crucial}

The notion of a {\em left-crucial permutation with respect to squares} can be defined similarly. Namely, a permutation is left-crucial with respect to squares if it is square-free but any extension of it to the left by an element gives a permutation that is not square-free. Taking into account that the reverse of a permutation  preserves the property of avoiding squares, it is equivalent to study right-crucial and left-crucial permutations (in particular, the numbers of these permutations of length $n$ are the same for any $n$).

The following proposition follows directly from the zigzag structure of square-free permutations described in Subsection \ref{square-free-perms}.

\begin{proposition}\label{prohibited-squares} If $\pi$ is a left-crucial (resp., right-crucial) permutation with respect to squares, and $\sigma=x\pi'$ (resp.,  $\sigma=\pi'x$) is its extension to the left (resp., right), then a prohibited square in $\sigma$ is either of length $4$, or of length multiple of~$8$. \end{proposition}

We prove the following proposition, whose analogue for right-crucial permutations is easy to obtain by applying the reverse operation to permutations that  turns, in particular, position 0 to position $n$.

\begin{theorem}\label{thm-16-24} Suppose $\pi$ is a left-crucial permutation with respect to squares. If extending $\pi$  to the left results in a square of length $16$ (that is, in $0\ldots 7=8\ldots15$) then there is no extension to the left of $\pi$ that results in a square of length $24$ (that is, in $0\ldots 11=12\ldots 23$).\end{theorem}

\begin{proof} We have that for some extension of $\pi$, $0\ldots 7=8\ldots 15$, and, in particular, $4\ldots 7 = 12\ldots 15$ (any subset of the former relation must have the same relative order). However, if another extension of $\pi$ gives a square of length 24, we would have to have that $0\ldots 7=12\ldots 19$, in particular, $4\ldots 7=16\ldots 19$, which leads to $12\ldots 15=16\ldots 19$, a contradiction (two consecutive factors would be equal in $\pi$, but $\pi$ is square-free). \end{proof}

\begin{proposition}\label{prop-4-squares} Suppose $\pi=\pi_1\pi_2\cdots\pi_n$ is a left-crucial permutation with respect to squares such that either $\pi_1<\pi_2<\pi_3$ or $\pi_1>\pi_2>\pi_3$. Then among its all possible $n+1$ extensions to the left, we will meet squares of at least four different lengths. \end{proposition}

\begin{proof} Assume that $\pi_1<\pi_2<\pi_3$; the other case can be considered analogously. Extending $\pi$ to the left by 1 will obviously give a square $S_1$ of length 4 formed by the pattern 12. Further, extending $\pi$ to the left by $\pi_1+1$, we get a permutation that begins with the pattern 2134, that is, with the four letters that are order-isomorphic to 2134. But then the respective square $S_2$ in this extension (there is one because $\pi$ is left-crucial) also begins with the pattern 2134, and thus, $S_2$ is different from $S_1$. Now, extending $\pi$ by $\pi_2+1$, we will obtain a permutation beginning with the pattern 3124, and thus the respective square $S_3$ also begins with the pattern 3124 and it is different from $S_1$ and $S_2$. Finally, extending $\pi$ by $n+1$ will give a permutation that begins with the pattern 4123, and thus the respective square $S_4$ also begins with the pattern 4123 and it is different from $S_1,S_2$ and $S_3$.\end{proof}

\begin{theorem}\label{thm-length-31} Suppose $\pi=\pi_1\pi_2\cdots\pi_n$ is a left-crucial permutation with respect to squares such that either $\pi_1<\pi_2<\pi_3$ or $\pi_1>\pi_2>\pi_3$. Then the length of $\pi$, $n$, is at least $31$. \end{theorem}

\begin{proof} We can assume that $\pi_1<\pi_2<\pi_3$; the other case can be considered analogously. By Proposition~\ref{prohibited-squares}, possible squares are of lengths 4, 8, 16, 24, 32, etc. By Proposition \ref{prop-4-squares} we have at least four different squares, and by Theorem~\ref{thm-16-24} we cannot have both of 16 and 24 length squares involved. This leads us to having at least one square of length at least 32. Thus, $\pi$ is of length at least 31.\end{proof}

\subsection{Bicrucial permutations with respect to squares}\label{sub-bicrucial}

In this subsection we define the notion of {\em bicrucial permutations with respect to squares} which was previously known \cite{AKPV}  as {\em maximal permutations with respect to squares}.  

A permutation is called bicrucial with respect to squares if it is both right-crucial and left-crucial.  We also call such permutations {\em bicrucial square-free permutations}. It was proved in \cite{AKPV}  that there exist bicrucial permutations with respect to
squares of odd lengths $8k+1, 8k+5, 8k+7$ for $k\ge 1$. Computer experiments show that the smallest bicrucial permutations with respect to squares are of length 9, and the number of such permutations of length $n$, for $n=9,10,\ldots, 20$, is
$$54, 0, 0, 0, 69856, 0, 2930016, 0, 40654860, 0, 162190472,0.$$ 
See  Table \ref{tab:maximal}: in some cases note that direct computations were too time consuming, and these numbers were computed from the numbers of symmetrically distinct solutions.  Recall that each symmetrically distinct $\pi$ gives four permutations through combinations of $c$ and $r$, unless $\pi = c(r(\pi))$ (equivalently, $\pi = r(c(\pi))$), in which case it gives two.    An interesting point in that table is that sometimes we could not solve a problem at one size but could solve it at a larger size.  While the smaller problem has fewer variables and constraints it nevertheless requires more search.   

\tabheader
 3 &            0 &            0 &            0 &            0 &            0 &            0 \\
 4 &            0 &            0 &            0 &            0 &            0 &            0 \\
 5 &            0 &            0 &            0 &            0 &            0 &            0 \\
 6 &            0 &            0 &            0 &            0 &            0 &            0 \\
 7 &            0 &           97 &            0 &            0 &            0 &            0 \\
 8 &            0 &          126 &            0 &            0 &            0 &            0 \\
 9 &           54 &          607 &           16 &           33 &            5 &            9 \\
10 &            0 &          351 &            0 &            0 &            0 &            0 \\
11 &            0 &         1665 &            0 &            0 &            0 &            0 \\
12 &            0 &         1422 &            0 &            0 &            0 &            0 \\
13 &        69856 &       298659 &        17548 &        48558 &          168 &          365 \\
14 &            0 &        63292 &            0 &            0 &            0 &            0 \\
15 &      2930016 &     14793584 &       732504 &      1981923 &            0 &            0 \\
16 &            0 &      3475684 &            0 &            0 &            0 &            0 \\
17 &     40654860 &    382563747 &     10165476 &     33999226 &         3522 &         8361 \\
18 &  0 & -
   &            0 &            0 &            0 &            0 \\
19 &  162190472
 & - 
 &     40547618 &    124608134 &            0 &            0 \\
20 & 0
 & -
 &            0 &            0 &            0 &            0 \\
21 & $\geq 1156065982$
& -
&$\geq  578032991$ 
 & $\geq 2091556603$
 &       287834 &       772800 \\
22 & 0 
 & -
 &            0 &            0 &            0 &            0 \\
23 & $\geq 1250325828$  & -
& $\geq 625162914$ & $\geq 1849967660$ &            0 &            0 \\
24 & ?
 & -
& $\geq 0$  & $\geq 1021275473$ &            0 &            0 \\
25 & $ \geq 28100262 $
 & - 
& $\geq 0$ & $ \geq 991823284 $  &     14050131 &     32022959 \\
26 & 0
 & -
 &            0 &     43972617 &            0 &            0 \\
\tabfooter{Enumerating bicrucial square-free permutations of length $n$. The first column of solutions  is the sequence A221990 in \cite{OEIS}.  From $n=18$ on we did not run experiments for the first column, but can make deductions from the other columns.  The strongest possible deductions are shown: a precise number based on the later columns where those runs completed; or $\geq$ to indicate that twice the number in the centre or right hand column is a lower bound; or `?' to indicate that no useful deduction can be made.  In the middle columns a $\geq$ indicates that that a run was started but did not finish in the time available: therefore the given numbers are lower bounds. We used the SSAC preprocessing option and then the dom/wdeg heuristic: for more details see Subsection~\ref{sec:cp}.}{tab:maximal}

Even though there do not exist any bicrucial square-free permutations of length $n=8k+3$ when $k=1$, there do exist such permutations of length $n=8k+3$ for $k=2$ and $k=3$, for example,
\begin{equation}\label{19-construction}
143289756(14)(11)(10)(17)(19)(16)(13)(15)(18)(12)
\end{equation}
and
\begin{equation}\label{27-construction}
312(27)(26)6(24)(25)54(11)(23)(12)8(10)(16)97(17)(21)(19)(14)(18)(20)(15)(13)(22),
\end{equation}
respectively. Thus, we have the following theorem that answers the respective question in \cite{AKPV}.

\begin{theorem} Regarding the case $n=8k+3$, there are no bicrucial square-free permutations of length $11$, while such permutations of lengths $19$ and $27$ exist. \end{theorem}

\begin{proof} The theorem follows from our computer experiments that, in particular, led us to discovery of permutations  (\ref{19-construction}) and (\ref{27-construction}). However, we will justify here why these permutations are not extendable to the left or to the right leaving to the Reader proving the fact that they are square-free.
\begin{itemize}
\item For the permutation (\ref{19-construction}),  extending it to the left by any element larger than 1 will give a square of length 4 formed by the pattern 21 (that is, 01=23 in this case following the notation in the beginning 
of this section), while extending it to the left by 1 we obtain a square of length 16 formed by the pattern 12543786 (that is, in this case $0\ldots 7=8\ldots 15$). On the other hand, extending the permutation to the right by an element larger than 12 we obtain a square of length 4 formed by the pattern 12 (that is, in this case (17)(18)=(19)(20)), while extending it to the right by any other element we obtain a square of length 8 formed by the pattern 3421 (that is, in this case $(13)\ldots(16)=(17)\ldots(20)$).
\item For the permutation (\ref{27-construction}),  extending it to the left by any element less than 4 we obtain a square of length 4 formed by the pattern 12 (that is, in this case 01=23), while extending the permutation to the left by any other element we obtain a square of length 8 formed by the pattern 4312 (that is, in this case $0\ldots 3=4\ldots 7$). As for extending to the right, doing so by an element less than 22 we obtain a square of length 4 formed by the pattern 21, while extending to the right by any other element we obtain a square of length 16 formed by the pattern 52463178 (that is, in this case $(13)\ldots(20)=(21)\ldots(28)$).
\end{itemize} 
\vspace{-8mm}\end{proof}

What we found to be especially striking is that bicrucial square-free permutations of even length do exist despite of what the data in Table \ref{tab:maximal} suggest. We record this result in the following theorem.

\begin{theorem} Bicrucial square-free permutations of even length exist. The shortest such permutation is of length $32$. 
\end{theorem}

\begin{proof} An example of a bicrucial square-free permutation of length 32 is 
\begin{small}
$$(28)(30)(31)(23)(22)(24)(29)(27)(19)(25)(26)(17)(13)(18)(21)(20)(14)(16)(32)879(15)(12)5(10)(11)31462.$$
\end{small}

The fact that the permutation above is proper was checked by computer. However, we demonstrate here that it is not extendable either to the left or to the right skipping explanation why it is square-free. Indeed, extending the permutation to the right by any element $>2$ we get a square of length 4 involving the pattern 12 (for example, extending to the right by 4, the rightmost four elements will be 5724, which is a square), while extending it by element 1 or 2, we will obtain $(26)\ldots(29)=(30)\ldots(33)$ (the respective pattern is 3421). See Table~\ref{32-left-extensions} for squares appearing while extending the permutation to the left.

\begin{table}
\begin{center}
\begin{tabular}{c|c|c}
Left-extension by element & Square length & Half of square pattern\\
\hline
<29 & 4 & 12 \\
29, 30 & 8 & 2134 \\
31 & 32 & (15)(12)(14)(16)768(13)(11)49(10)2135 \\
32, 33 & 16 & 84672135 \\
\end{tabular}
\caption{Squares appearing while extending the permutation of length 32 to the left.}\label{32-left-extensions}
\end{center}
\end{table}
 
The fact that the permutation of length 32 is the shortest possible out of bicrucial square-free permutations of even length was checked by computer. However, we provide here an argument justifying this fact. 

Let $\pi$ be a bicrucial permutation of even length. The key observation is that because of the zigzag structure described in Subsection \ref{square-free-perms}, $\pi$ either begins or ends with three elements in a monotonic order (either increasing or decreasing). If necessary, we can apply the reverse operation to be able to assume without loss of generality that $\pi$ begins either with three increasing or three decreasing elements. But then by Theorem \ref{thm-length-31}, the length of $\pi$ is at least 32 (it must be even). \end{proof}

To see more  properties of bicrucial square-free permutations, we enumerate symmetrically distinct such permutations, that is, ensuring that for each  bicrucial permutations $\pi$, the  permutations $\pi$, $r(\pi)$, $c(\pi)$ and $r(c(\pi))=c(r(\pi))$ are collectively counted exactly once. The results are presented in Table~\ref{tab:maximal} in the central columns. We also find the number of symmetrically distinct bicrucial square-free permutations on the set of permutations invariant under taking the composition of the reverse and complement operations, that is, the set of permutations $\pi$ such that $\pi=r(c(\pi))=c(r(\pi))$. Note that if $\pi$ has this property then so does the permutation $c(\pi)$: again we only count one of these permutations. These results are presented in Table~\ref{tab:maximal} in the right hand columns.

Note that each symmetrically distinct bicrucial square-free permutation gives four different permutations, with the exception of permutations unchanged under the composition of taking the reverse and complement. These latter permutations give two different permutations.  
Once again, by the inclusion-exclusion principle, for each $n$, the number of bicrucial square-free permutations is four times the number in the central column of Table~\ref{tab:maximal} minus twice the number in the right hand column of
Table~\ref{tab:maximal}.   This serves as a check for $k=9,13,15, 17$ and allows us to extend the sequence A221990 in \cite{OEIS} for 
 $n=19$, with $4\cdot 40,547,618=162,190,472$, which is recorded in Table~\ref{tab:maximal}.  We are sometimes able to provide all solutions up to symmetry when we cannot enumerate them in full, simply because the search space is (approximately) four times smaller.

\subsection{$P$-crucial and S-crucial permutations with respect to squares}\label{P-crucial-S-crucial}

The following notion is defined for the first time in this paper. 

Given a set of non-negative integers $P$ and a set of prohibitions $S$, a permutation is called {\em $P$-crucial with respect to $S$}, if it avoids the prohibitions but {\em any} of its extensions in position $i$ results in a permutation containing a prohibition from $S$, whenever $i\in P$. Sets $P$ and $S$ can either be finite or infinite. In particular, $S$ can be a set of prohibited factors, for example, the set of all squares considered in this paper.  If $P=\{0\}$ then $P$-crucial permutations are just left-crucial permutations, and thus we deal with a generalisation of this notion. However, to have the most general definition, in particular, generalising the notion of right-crucial and bicrucial permutations, we allow $P$ to be defined using the length of permutations. For example, we can say that $P$ refers to positions $1$, $n-3$ and $n-2$ in a permutation of length $n$. Similarly, right-crucial permutations correspond to the case $P=\{n\}$, while bicrucial permutations correspond to the case $P=\{0,n\}$.

{\em S-crucial permutations} are $P$-crucial permutations with $P=\{0,1,2,\ldots\}$, that is, any extension of an S-crucial permutation in any position will lead to an occurrence of a prohibition. S in ``S-crucial'' stands for ``Super''. 

The notions of $P$-crucial and S-crucial permutations with respect to a given set of prohibitions can be easily extended to  the case of words, which is not in the scope of this paper. Our immediate interest is in S-crucial  permutations with respect to squares and in the question whether such permutations exist. It is immediate from definitions that 

\begin{center}
\begin{tabular}{c}
S-crucial permutations  $\subseteq$ \\
bicrucial permutations  = right-crucial permutations $\cap$ left-crucial permutations.
\end{tabular}
\end{center}

\subsection{S-crucial permutations with respect to squares}\label{sub-S-crucial}

Taking into account the double zigzag structure of square-free permutations described in Subsection~\ref{square-free-perms}, one sees that in order to test if a given square-free permutation is S-crucial or not, we only need to check what happens when extending the permutation in positions $i=0,1,n-1,n$. Indeed, inserting the new element in any other position will obviously break the double zigzag structure (no matter what the element will be) and therefore will cause the obtained permutation to contain a square.  Thus, in the case of prohibited squares, S-crucial permutations accept an equivalent definition as $P$-crucial permutations with respect to squares, where $P=\{0,1,n-1,n\}$ for permutations of length $n$.   

\begin{theorem} S-crucial permutations with respect to squares exist, and the shortest such permutation is of length $17$. There are $1568$ S-crucial permutations with respect to squares of length $17$. \end{theorem}

\begin{proof} The following permutation of length $17$ is S-crucial with respect to squares:
$$2 4 3 1 5 (11) (10) 6 9 (12) 8 7 (13) (17) (15) (14) (16),$$
which can be checked by confirming that it is square-free but all of its extensions in positions 0, 1, 15 and 16 produce squares. Our proof otherwise relies on computer experiments reported in Table~\ref{tab:01n1n} showing that no S-crucial permutation with respect to squares exist of length less than 17, and the total for $n=17$.  
\end{proof}

\begin{table}
\begin{center}
\begin{tabular}{c|r|r|c|c}
$P$ & min length & \# of min length & Type of permutations  & Table\\
\hline
$\{0\}$, $\{n\}$ & 7 & 60 & left- \& right-crucial (Subsec. \ref{sub-right-crucial-perm}) & \ref{tab:right-crucial}\\
$\{1\}$, $\{n-1\}$ & 7 & 82 & &  \ref{tab:1}  \\
$\{0,1\}$, $\{n-1,n\}$  &15 & 54854 & & \ref{tab:01} \\
$\{0,n-1\}$, $\{1,n\}$  & 7 & 20 & & \ref{tab:0n1} \\
$\{0,n\}$ & 9 & 54 & bicrucial (Subsec. \ref{sub-bicrucial})&\ref{tab:maximal}\\
$\{1,n-1\}$ & 7 & 18 & &\ref{tab:1n1}\\
$\{0,1,n-1\}$, $\{1,n-1,n\}$  & 16 & 553428 & &\ref{tab:01n1}\\
$\{0,1,n\}$, $\{0,n-1,n\}$  & 17 & 550976 & &\ref{tab:01n}\\
$\{0,1,n-1,n\}$ & 17 & 1568 & S-crucial (Subsec. \ref{sub-S-crucial})&\ref{tab:01n1n}\\
\end{tabular}
\caption{$P$-crucial permutations with respect to squares.}\label{P-crucial}
\end{center}
\end{table}

\subsection{More on $P$-crucial permutations with respect to squares}

As is mentioned above, S-crucial permutations coincide with $\{0,1,n-1,n\}$-crucial permutations, because extending square-free permutations in positions different from those in the set  $\{0,1,n-1,n\}$ will automatically give a square. Thus, for any $P$,  $P$-crucial permutations with respect to squares are equivalent to $S$-crucial permutations with respect to squares for some $S\subseteq P$. The case of empty $S$ is not interesting, and thus we essentially have 15 classes to consider of $P$-crucial permutations with respect to squares. Moreover, because of the reverse operation, some of these 15 classes are equivalent to study. For example, as we know, studying left-crucial permutations ($P=\{0\}$) is equivalent to studying right-crucial permutations ($P=\{n\}$). In Table \ref{P-crucial} we summarise our knowledge on $P$-crucial permutations with respect to squares that is based on computer experiments.

\section{Generating (right-,bi)crucial permutations with respect to squares}\label{generating-perms}

In Subsection \ref{original-approach} we discuss our original approach to deal with square-free permutations, and in Subsection~\ref{encoding-orderings} we discuss an optimisation via encoding orderings. 

\subsection{An approach to generate square-free permutations}\label{original-approach}

We define the \textit{left parent} (resp., right parent) of a permutation $\pi\in S_n$, the set of permutations of length $n$, as the unique permutation in $S_{n-1}$ order-isomorphic to the sequence of the first (resp., last) $n-1$ letters of $\pi$. The \textit{left children} (resp., right children) of $\sigma\in S_{n-1}$  are all $\pi\in S_n$ such that $\sigma$ is a right (resp., left) parent of $\pi$.

Clearly the parents of a square-free permutation are square-free; a right-crucial permutation is one with no square-free right children and a bicrucial permutation is one with no square-free children at all. 

Our construction algorithm is iterative. At each step we assume that we have a data structure representing the square-free permutations of lengths $n-2$ and connecting each of them with all of its children (of length $n-1$). This necessarily includes all square-free permutations of length $n-1$. We now compute all the square-free permutations of length $n$ as follows. For each square-free $\alpha$ of length $n-2$, each left child $\sigma$ of $\alpha$ and each right child $\tau$ of $\alpha$ we form all permutations $\pi$ which have $\sigma$ as left parent and $\tau$ as right parent. There are either one or two of these. Such a $\pi$ can only fail to be square-free if $n$ is even, greater than 2, and $\pi$ is a repeat of a pattern of length $n/2$. Any shorter repeated pattern would be contained in at least one of $\sigma$ and $\tau$. This condition can be checked efficiently. We record all $\pi$ which pass this check, both as square-free permutations of length $n$ and as children of $\sigma$ and $\tau$, thereby preparing for the next iteration. A suitable data structure for the permutations of length zero and one can be constructed ``by hand'' to start the process.

After this computation, we can easily read off from our data structure all the square-free permutations of length $n$ and the right-crucial and bicrucial permutations of length $n-1$.   

We used the Computational Algebra system GAP \cite{GAP4}.  The calculation can be completed up to $n=18$ in a few hours on a computer with 512GB RAM.  Memory, rather than time, is the main barrier to continuing. However, in order to achieve our results for longer permutations, we had to come up with a different idea of generating permutations in question, which is discussed in the following subsection.

\subsection{Representing Permutations in Constraint Programming}\label{encoding-orderings}
\label{sec:cp}

Constraint programming is a means of solving finite domain combinatorial problems
\cite{handbook-cp}.  
It is the subject of much research and many fast and efficient solvers are available. In particular we use Minion, a solver developed at St Andrews \cite{minion}.   As well as fast solving, it is critical that problems are \emph{modelled} effectively.  
Modelling is the process of converting an abstract specification of a problem to a constraint satisfaction problem that 
can be searched effectively.   To help with the modelling process we use Savile Row, an automated modelling assistant
\cite{savilerow-website}.  Savile Row allows us to express models in a higher level language than the input language of Minion.

The question we have in front of us is how to model permutations in a way that can be represented in a constraint solver 
and reasoned with effectively in the context of square-free permutations.   This is an unusual application, because the properties of the permutation being used are not the normal ones. Normally in constraints one is interested in the value of each element of the permutation, and using them in some way.   Here, we are interested only in the 
relative ordering of elements.  So, we do not need to know efficiently that  the permutation maps (say) 4 to 3, but we must be able to detect efficiently that the partial permutation (say) 2813 represents the same ordering as 5824 but is different from 1823.   Standard representations of permutations in constraint solving are not amenable to this form of reasoning. 

Therefore, we constructed a new model of permutations in constraints which is efficient for the current application.
Assume that we have permutations stored in two arrays $A$ and $B$ (although in practice they may be two segments of the same array).  
The encoding idea is  straightforward.  
An ordering is uniquely defined by the binary inequalities between elements. 
To exploit this we use ``reification", see e.g. \cite{reification}.  The reified value of a particular constraint 
is its truth value treated as a boolean: i.e. the reified value is 1 if the constraint is false and 0 if the constraint is true.
We  introduce a new boolean variable for the reified value of each 
inequality.  That is, we have a variable $a_{i,j}$ such that 
$a_{i,j} = T \equiv A[i] < A[j]$.  Similarly,
$b_{i,j} = T \equiv B[i] < B[j]$.  
Then we have the simple observation that the ordering of $A$ and $B$ with respect
to $<$ is the same if and only if $\forall {i,j: } a_{i,j} = b_{i,j}$.
This forms the basis of our encoding of square-free permutations: stating that no adjacent parts of permutations may be the same.  For cruciality constraints we state that for each possible value added at the relevant position, at least one pair of identical parts of permutations must result.

In our experiments we used Minion version 1.6.   In most cases we used it with standard settings.  However in two experiments we gained significant value from other settings.  In Tables~\ref{tab:maximal} and~\ref{tab:01n1n} we 
used a specific preprocessing option and search heuristic. The SSAC preprocessing option was given to Minion, which performs Singleton-Singleton-(generalised)Arc-Consistency before search starts.   SAC (singleton arc consistency) sets each variable  to each possible value in turn, and if this leads to failure we can undo this and assert the variable cannot take this value \cite{singleton}.   SSAC doubles this, i.e., setting each variable to each possible value and then calling SAC, removing values after failure.  This is an extremely expensive preprocessing step but we found sometimes that it reduced search so much it was worthwhile.   It also combines well with the  dom/wdeg heuristic available in Minion, which we used in these cases.   

In all cases, if $\pi$ satisfied the given property, $c(\pi)$ must.  To calculate numbers of solutions up to symmetry, we used one of two approaches.  If $c(\pi)$ is the only symmetry we can simply halve the number of solutions found in complete search, since a permutation of length $>1$ is never self-complementing.   In some cases $r(\pi)$ is a symmetry: specifically for square-free permutations, and P-crucial permutations for $P = \{0,n\}$, $\{1,n-1\}$ or $\{0,1,n-1,n\}$.  
In these cases we excluded all but a canonical solution from each equivalence class, using the standard `lex leader' approach \cite{lexleader}.   This is efficient because of the use of specialised constraints for lexicographical ordering 
\cite{gaclex}. 

We have not reported run-times in detail, but as an example, the $n=17$ run in Table~\ref{tab:maximal} required  54,858 cpu seconds, just under 15.25 cpu hours.  This is a rate of just under 7,000 nodes per second.    A key advantage of constraint programming is the massively reduced RAM requirements compared to Section~\ref{original-approach}.    
Because of this we are able to solve much larger problems: for example in Table~\ref{tab:01n1n} we are able to settle all questions for $n=22$ and one for $n=26$, compared to a maximum of $n=18$ with  the earlier approach. 

\section{Open questions and some conjectures}\label{open}

It would be interesting to find the missing (exact) solutions in Tables~\ref{tab:maximal} and~\ref{tab:01n1n}.  Also, we would like to state the following conjectures.

\begin{conjecture}\label{bicrucials2} There exist bicrucial permutations with respect to squares of length $8k+3$ for $k\geq 2$.\end{conjecture}

\begin{conjecture}\label{even-length-bicrucial} There exist arbitrary long bicrucial square-free permutations of even length. \end{conjecture}

If Conjecture \ref{even-length-bicrucial} is true, it is interesting to know whether such permutations exist for each even length greater than 30. 

\begin{conjecture}\label{long-S-crucial} There exist arbitrary long S-crucial permutations with respect to squares. \end{conjecture}

Table~\ref{tab:01n1n} gives our empirical investigations into this question.   While there were no S-crucial permutations of length 18, 19, 20 or 22, we did find (at least) 144,586 symmetrically distinct solutions at $n=21$.   

Finally, a general program of research here is the study of classes of $P$-crucial permutations with respect to a given set of prohibitions $S$ (different from squares considered in this paper). One can study such objects in the way we study, say, right-crucial and bicrucial permutations, that is, to try to classify lengths for which respective $P$-crucial permutations exist, or at least to try to show that arbitrary long such permutations exist. Also, considering words instead of permutations for various (natural) sets $P$ and $S$ is yet another interesting research direction. 

\section*{Acknowledgements}

We are extremely grateful to Chris Jefferson for suggesting that square-free permutations may be appropriate for solution using constraint programming.  This research is supported in part by EPSRC grants EP/G055181/1 and EP/H004092/1, for which we are very grateful.

\appendix

\section{Additional results}

In this appendix we include full tables of results for all forms of cruciality not given in full detail in the main part of the paper.   
These tables are referred to in the main text by Table~\ref{P-crucial}. The tables in this Appendix take one of two forms.
As mentioned in the main text, for some sets $P$ we have that if $\pi$ is $P$-crucial then it is guaranteed that $r(\pi)$ is also.  In these cases a separate run is required to identify the number of symmetrically distinct solutions.  Tables \ref{tab:1n1} and \ref{tab:01n1n} are of this form.   
Where reversal symmetry is not guaranteed (the remaining tables), the tables are different in two ways.  
First, no separate run is necessary to count symmetrically distinct solutions: the total is always half the total number of solutions.  Second, the table can be computed in one of two ways.   For example, Table~\ref{tab:1} shows numbers of 
$\{1\}$-and $\{n-1\}$-crucial permutations.   While it is guaranteed that the number of solutions of each type is the same, the number of nodes searched may be different depending on which type was run.  For precision, we indicate in each table which node count we are reporting.  For example, in  Table~\ref{tab:1}  we give node counts for $\{1\}$-crucial permutations.   The choice of which case is reported is arbitrary.   

\tabheadershort
 3 &            0 &            0 &            0 &            0 &            0 \\
 4 &            0 &            6 &            0 &            0 &            5 \\
 5 &            0 &           10 &            0 &            0 &            7 \\
 6 &            0 &           15 &            0 &            0 &           13 \\
 7 &           82 &          351 &           41 &            3 &           28 \\
 8 &          272 &         1862 &          136 &            0 &           25 \\
 9 &          766 &         5955 &          383 &            0 &          112 \\
10 &         3788 &        19687 &         1894 &            0 &          248 \\
11 &        14096 &        75932 &         7048 &           58 &          617 \\
12 &        74568 &       322940 &        37284 &            0 &           61 \\
13 &       281232 &      1358128 &       140616 &            0 &         2894 \\
14 &      2026184 &      7636544 &      1013092 &            0 &          848 \\
15 &      9430962 &     42623572 &      4715481 &          961 &        13787 \\
16 &     79497550 &    400446913 &     39748775 &            0 &          113 \\
17 &  422657308 &   2274985904 &    211328654   &           28 &       101644 \\
\tabfooter{Results for  $\{1\}$ and $\{n-1\}$-crucial permutations.  Node counts from former.}{tab:1}

\tabheadershort
 3 &            0 &            0 &            0 &            0 &            0 \\
 4 &            0 &            0 &            0 &            0 &            0 \\
 5 &            0 &            0 &            0 &            0 &            0 \\
 6 &            0 &            0 &            0 &            0 &            0 \\
 7 &            0 &          102 &            0 &            0 &           16 \\
 8 &            0 &          161 &            0 &            0 &           25 \\
 9 &            0 &          244 &            0 &            0 &           34 \\
10 &            0 &          351 &            0 &            0 &          232 \\
11 &            0 &          485 &            0 &            0 &           46 \\
12 &            0 &          649 &            0 &            0 &           61 \\
13 &            0 &          846 &            0 &            0 &          122 \\
14 &            0 &         1079 &            0 &            0 &          845 \\
15 &        54854 &      1020814 &        27427 &            0 &         4113 \\
16 &       722114 &     24479482 &       361057 &            0 &          113 \\
17 &      5144632 &    118675744 &      2572316 &           28 &        53501 \\
\tabfooter{Results for $\{0,1\}$ and $\{n-1,n\}$-crucial permutations.  Node counts from former.}{tab:01}

\tabheadershort
 3 &            0 &            0 &            0 &            0 &            0 \\
 4 &            0 &            0 &            0 &            0 &            0 \\
 5 &            0 &            0 &            0 &            0 &            0 \\
 6 &            0 &            0 &            0 &            0 &            0 \\
 7 &           20 &          111 &           10 &            0 &           16 \\
 8 &           96 &          530 &           48 &            0 &           25 \\
 9 &            0 &         1266 &            0 &            0 &           39 \\
10 &         1444 &         7815 &          722 &            0 &          232 \\
11 &            0 &         1294 &            0 &            0 &           46 \\
12 &        10080 &        66223 &         5040 &            0 &           61 \\
13 &            0 &        98871 &            0 &            0 &         1135 \\
14 &            0 &       351920 &            0 &            0 &          845 \\
15 &         2988 &      1883376 &         1494 &            0 &         4113 \\
16 &     25781024 &    160519095 &     12890512 &            0 &          113 \\
17 &      2138998 &    294150147 &      1069499 &           28 &        56458 \\
\tabfooter{Results for $\{0,n-1\}$ and $\{1,n\}$-crucial permutations.  Node counts from latter.}{tab:0n1}

\tabheader
 3 &            0 &            0 &            0 &            0 &            0 &            0 \\
 4 &            0 &            6 &            0 &            3 &            0 &            3 \\
 5 &            0 &           10 &            0 &            3 &            0 &            3 \\
 6 &            0 &           15 &            0 &            6 &            0 &            6 \\
 7 &           18 &          225 &            6 &           68 &            3 &           12 \\
 8 &            0 &         1385 &            0 &          347 &            0 &           10 \\
 9 &            0 &         4677 &            0 &         1288 &            0 &           52 \\
10 &            0 &        13971 &            0 &         2936 &            0 &          130 \\
11 &         8972 &        64562 &         2272 &        16749 &           58 &          294 \\
12 &            0 &       199164 &            0 &        64118 &            0 &           21 \\
13 &       281232 &      1216051 &        70308 &       371785 &            0 &         1403 \\
14 &            0 &      3754582 &            0 &      1312833 &            0 &          434 \\
15 &      3094458 &     32430311 &       774095 &      7744341 &          961 &         6828 \\
16 &      1194800 &    257503934 &       298700 &     67920867 &            0 &           36 \\
17 &{6056996}  &      1714652389 &      1514263 &    524235669 &           28 &        52973 \\
\tabfooter{Results for $\{1,n-1\}$-crucial permutations.}{tab:1n1}

\tabheadershort
 3 &            0 &            0 &            0 &            0 &            0 \\
 4 &            0 &            0 &            0 &            0 &            0 \\
 5 &            0 &            0 &            0 &            0 &            0 \\
 6 &            0 &            0 &            0 &            0 &            0 \\
 7 &            0 &          102 &            0 &            0 &           16 \\
 8 &            0 &          161 &            0 &            0 &           25 \\
 9 &            0 &          244 &            0 &            0 &           34 \\
10 &            0 &          351 &            0 &            0 &          232 \\
11 &            0 &          485 &            0 &            0 &           46 \\
12 &            0 &          649 &            0 &            0 &           61 \\
13 &            0 &          846 &            0 &            0 &          122 \\
14 &            0 &         1079 &            0 &            0 &          845 \\
15 &            0 &       841776 &            0 &            0 &         4113 \\
16 &       553428 &     23833969 &       276714 &            0 &          113 \\
17 &         5424 &    107571076 &         2712 &           28 &        53437 \\
\tabfooter{Results for $\{0,1,n-1\}$ and $\{1,n-1,n\}$-crucial permutations.  Node counts from former.}{tab:01n1}

\tabheadershort
 3 &            0 &            0 &            0 &            0 &            0 \\
 4 &            0 &            0 &            0 &            0 &            0 \\
 5 &            0 &            0 &            0 &            0 &            0 \\
 6 &            0 &            0 &            0 &            0 &            0 \\
 7 &            0 &           83 &            0 &            0 &           16 \\
 8 &            0 &          126 &            0 &            0 &           25 \\
 9 &            0 &          479 &            0 &            0 &           43 \\
10 &            0 &          351 &            0 &            0 &          232 \\
11 &            0 &         1665 &            0 &            0 &           46 \\
12 &            0 &         1422 &            0 &            0 &           61 \\
13 &            0 &       110752 &            0 &            0 &         1243 \\
14 &            0 &        63292 &            0 &            0 &          845 \\
15 &            0 &      7381558 &            0 &            0 &         4113 \\
16 &            0 &      3471394 &            0 &            0 &          113 \\
17 &       550976 &    282598708 &       275488 &           28 &        56784 \\
\tabfooter{Results for $\{0,1,n\}$ and $\{0,n-1,n\}$-crucial permutations.  Node counts from latter.}{tab:01n}

\tabheader

 3 &            0 &            0 &            0 &            0 &            0 &            0 \\
 4 &            0 &            0 &            0 &            0 &            0 &            0 \\
 5 &            0 &            0 &            0 &            0 &            0 &            0 \\
 6 &            0 &            0 &            0 &            0 &            0 &            0 \\
 7 &            0 &            0 &            0 &            0 &            0 &            0 \\
 8 &            0 &            0 &            0 &            0 &            0 &            0 \\
 9 &            0 &            0 &            0 &            0 &            0 &            0 \\
10 &            0 &            0 &            0 &            0 &            0 &            0 \\
11 &            0 &            0 &            0 &            0 &            0 &            0 \\
12 &            0 &            0 &            0 &            0 &            0 &            0 \\
13 &            0 &            0 &            0 &            0 &            0 &            0 \\
14 &            0 &            0 &            0 &            0 &            0 &            0 \\
15 &            0 &            0 &            0 &            0 &            0 &            0 \\
16 &            0 &            0 &            0 &            0 &            0 &            0 \\
17 &         1568 &      9214495 &          406 &          823 &           28 &           55 \\
18 &            0 &            0 &            0 &            0 &            0 &            0 \\
19 &            0 &     17819710 &            0 &            0 &            0 &            0 \\
20 &            0 &            0 &            0 &            0 &            0 &            0 \\
21 & $\geq$ 289172& {\emph{timeout}}  &      $\geq 144586$ &    $\geq  5795227$ &          \multicolumn{2}{c}{\emph{timeout}} \\
22 &            0 &            0 &            0 &            0 &            0 &            0 \\
23 &             \multicolumn{2}{c|}{\emph{timeout}}  &         \multicolumn{2}{c|}{\emph{timeout}}  &            0 &            0 \\
24 &              \multicolumn{2}{c|}{\emph{timeout}}  &         \multicolumn{2}{c|}{\emph{timeout}} &            0 &            0 \\
25 &              \multicolumn{2}{c|}{\emph{timeout}}  &         \multicolumn{2}{c|}{\emph{timeout}}&      \multicolumn{2}{c}{\emph{timeout}}  \\
26 &              \multicolumn{2}{c|}{\emph{timeout}}  &         \multicolumn{2}{c|}{\emph{timeout}}&            0 &            0 \\

\tabfooter{Results for $\{0,1,n-1, n\}$-crucial permutations (also called S-crucial permutations in the main text).   A timed out run is one which failed to find any solutions before 6 hours cpu time were used on our machine.  Note that problems do not get monotonically harder, presumably because some sizes, e.g. 22, are intrinsically easier to search than other sizes.   Interestingly, 
 no solutions were found for 21 in regular search. However, because  144,586 were found up to symmetry before the timeout: all of these and their complements  are S-crucial permutations so we obtain the lower bound given in the first column. These experiments were run with SSAC preprocessing and dom/wdeg heuristic.  }{tab:01n1n}

\end{document}